\documentclass{article}
\usepackage{amsmath, amsthm}
\usepackage{amssymb}
\newtheoremstyle{theorem}
  {10pt}          
  {10pt}  
  {\sl}  
  {\parindent}     
  {\bf}  
  {. }    
  { }    
  {}     
\newtheorem{thm}{Theorem}[section]

\newtheorem{lem}[thm]{Lemma}

\newtheorem{defn}{Definition}[section]

\newtheorem{rem}{Remark}[section]
\begin{document}
\title{\large\bf
Fractional differences and sums with binomial coefficients}
\author{\small \bf T. Abdeljawad $^a$, D. Baleanu $^{a,b}$ , F. Jarad $^{a}$ , R. Agarwal $^{c}$  \\ {\footnotesize $^a$ Department of
Mathematics, \c{C}ankaya University, 06530, Ankara, Turkey}\\
{\footnotesize $^b$ Institute of Space Sciences, Magurele-Bucharest,Romania}\\
{\footnotesize $^c$ Department of Mathematics, Texas  A $\&$ M University - Kingsville, USA}}
\date{}
\maketitle
{\footnotesize {\noindent\bf Abstract.} In fractional calculus there are two approaches to obtain fractional derivatives. The first approach  is by  iterating the integral and then defining a fractional order by using Cauchy formula to obtain Riemann fractional integrals and derivatives. The second approach is by iterating the derivative and then defining a fractional order by making use of the binomial theorem to obtain Gr\"{u}nwald-Litnikov fractional derivatives. In this article we formulate the delta and  nabla discrete versions for left and right fractional integrals and derivatives representing the second approach. Then, we use the discrete version of the Q-operator and some discrete fractional dual identities to prove that the presented fractional differences and sums coincide with the discrete Riemann ones describing the first approach.


\section{Introduction and preliminaries}
Fractional calculus (FC) is developing very fast in both theoretical and applied aspects. As a result FC is used intensively and successfully in the last few decades to describe the  anomalous processes which appear in complex systems \cite{book1, book2,book3,book4,book5,book6}. The complexity of the real world phenomena is a great source of inspiration for the researchers to invent new fractional tools which will be able to dig much dipper into the mysteries of the mother nature. Historically the FC passed thorough different periods  of evolutions and it started to face very recently a new provocation: how to formulate properly its discrete counterpart \cite{Gray,Miller,Ferd1,Ferd2,Ferd3,Ferd4,Gronwall,Thabet1,Thabet2,Holm,Gdelta,Gnabla, Gfound, Nuno}. At this stage we have to stress on the fact that in the classical discrete equations their roots are based on the functional difference equations, therefore the natural question is to find the generalization of these equations to the fractional case. In other words we will end up  with generalizations of the basic operators occurring in standard difference equations. As it was expected there were several attempts to do this generalization as well as to apply this new techniques to investigate the dynamics of some complex processes. In recent years the discrete counterpart of the fractional Riemann-Liouville, Caputo, were investigated mainly thinking how to apply  techniques from the time scales calculus to the expressions of the fractional operators. Despite of the beauty of the obtained results one simple question arises: can we obtain the same results from a new point of view which is more simpler and more intuitive?
Having all above mentioned thinks in mind we  are going to use the binomial theorem in order to get Gr\"{u}nwald-Litnikov fractional derivatives. After that we proved that the results obtained  coincide with the ones obtained by the discretization of the Riemann-Liouville operator. In this manner we believe that it becomes more clear what the fractional difference equations bring new in description of the  related complex phenomena described.

For a natural number $n$, the fractional polynomial is defined by,

 \begin{equation} \label{fp}
 t^{(n)}=\prod_{j=0}^{n-1} (t-j)=\frac{\Gamma(t+1)}{\Gamma(t+1-n)},
 \end{equation}
where $\Gamma$ denotes the special gamma function and the product is
zero when $t+1-j=0$ for some $j$. More generally, for arbitrary
$\alpha$, define
\begin{equation} \label{fpg}
t^{(\alpha)}=\frac{\Gamma(t+1)}{\Gamma(t+1-\alpha)},
\end{equation}
where the convention that division at pole yields zero.
Given that the forward and backward difference operators are defined
by
\begin{equation} \label{fb}
\Delta f(t)=f(t+1)-f(t)\texttt{,}~\nabla f(t)=f(t)-f(t-1)
\end{equation}
respectively, we define iteratively the operators
$\Delta^m=\Delta(\Delta^{m-1})$ and $\nabla^m=\nabla(\nabla^{m-1})$,
where $m$ is a natural number.

 Here are some properties of the  factorial function.

\begin{lem} \label{pfp} (\cite{Ferd1})
Assume the following factorial functions are well defined.

(i) $ \Delta t^{(\alpha)}=\alpha  t^{(\alpha-1)}$.

(ii) $(t-\mu)t^{(\mu)}= t^{(\mu+1)}$, where $\mu \in \mathbb{R}$.

(iii) $\mu^{(\mu)}=\Gamma (\mu+1)$.

(iv) If $t\leq r$, then $t^{(\alpha)}\leq r^{(\alpha)}$ for any
$\alpha>r$.

(v) If $0<\alpha<1$, then $ t^{(\alpha\nu)}\geq
(t^{(\nu)})^\alpha$.

(vi) $t^{(\alpha+\beta)}= (t-\beta)^{(\alpha)} t^{(\beta)}$.
\end{lem}

Also, for our purposes we list down the following two properties, the proofs of which are straightforward.

\begin{equation} \label{ou1}
\nabla_s (s-t)^{(\alpha-1)}=(\alpha-1)(\rho(s)-t)^{(\alpha-2)}.
\end{equation}

\begin{equation} \label{ou2}
\nabla_t
(\rho(s)-t)^{(\alpha-1)}=-(\alpha-1)(\rho(s)-t)^{(\alpha-2)}.
\end{equation}

For the sake of the nabla fractional calculus we have the following definition

\begin{defn} \label{rising}(\cite{Adv,Boros,Grah,Spanier})

(i) For a natural number $m$, the $m$ rising (ascending) factorial of $t$ is defined by

\begin{equation}\label{rising 1}
    t^{\overline{m}}= \prod_{k=0}^{m-1}(t+k),~~~t^{\overline{0}}=1.
\end{equation}

(ii) For any real number the $\alpha$ rising function is defined by
\begin{equation}\label{alpharising}
 t^{\overline{\alpha}}=\frac{\Gamma(t+\alpha)}{\Gamma(t)},~~~t \in \mathbb{R}-~\{...,-2,-1,0\},~~0^{\mathbb{\alpha}}=0
\end{equation}

\end{defn}

Regarding the rising factorial function we observe the following:

(i) \begin{equation}\label{oper}
    \nabla (t^{\overline{\alpha}})=\alpha t^{\overline{\alpha-1}}
\end{equation}

 (ii)
 \begin{equation}\label{oper2}
    (t^{\overline{\alpha}})=(t+\alpha-1)^{(\alpha)}.
 \end{equation}

(iii)
\begin{equation}\label{oper3}
   \Delta_t (s-\rho(t))^{\overline{\alpha}}= -\alpha  (s-\rho(t))^{\overline{\alpha-1}}
\end{equation}

\textbf{Notation}:
\begin{enumerate}
\item[$(i)$] For a real $\alpha>0$, we set $n=[\alpha]+1$, where $[\alpha]$ is the greatest integer less than $\alpha$.

\item[$(ii)$] For real numbers $a$ and $b$, we denote $\mathbb{N}_a=\{a,a+1,...\}$ and $~_{b}\mathbb{N}=\{b,b-1,...\}$.

\item[$(iii)$]For $n \in \mathbb{N}$ and real $a$, we denote
$$ _{\circleddash}\Delta^n f(t)\triangleq (-1)^n\Delta^n f(t).$$
\item[$(iv)$]For $n \in \mathbb{N}$ and real $b$, we denote
                   $$ \nabla_{\circleddash}^n f(t)\triangleq (-1)^n\nabla^n f(t).$$
\end{enumerate}

\begin{defn} \label{fractional sums}
Let $\sigma(t)=t+1$ and $\rho(t)=t-1$ be the forward and backward jumping operators, respectively. Then

(i) The (delta) left fractional sum of order $\alpha>0$ (starting from $a$) is defined by:
\begin{equation}\label{dls}
    \Delta_a^{-\alpha} f(t)=\frac{1}{\Gamma(\alpha)} \sum_{s=a}^{t-\alpha}(t-\sigma(s))^{(\alpha-1)}f(s),~~t \in \mathbb{N}_{a+\alpha}.
\end{equation}

(ii) The (delta) right fractional sum of order $\alpha>0$ (ending at  $b$) is defined by:
\begin{equation}\label{drs}
   ~_{b}\Delta^{-\alpha} f(t)=\frac{1}{\Gamma(\alpha)} \sum_{s=t+\alpha}^{b}(s-\sigma(t))^{(\alpha-1)}f(s)=\frac{1}{\Gamma(\alpha)} \sum_{s=t+\alpha}^{b}(\rho(s)-t)^{(\alpha-1)}f(s),~~t \in ~_{b-\alpha}\mathbb{N}.
\end{equation}

(iii) The (nabla) left fractional sum of order $\alpha>0$ (starting from $a$) is defined by:
\begin{equation}\label{nlf}
  \nabla_a^{-\alpha} f(t)=\frac{1}{\Gamma(\alpha)} \sum_{s=a+1}^t(t-\rho(s))^{\overline{\alpha-1}}f(s),~~t \in \mathbb{N}_{a+1}.
\end{equation}

(iv)The (nabla) right fractional sum of order $\alpha>0$ (ending at $b$) is defined by:
\begin{equation}\label{nrs}
   ~_{b}\nabla^{-\alpha} f(t)=\frac{1}{\Gamma(\alpha)} \sum_{s=t}^{b-1}(s-\rho(t))^{\overline{\alpha-1}}f(s)=\frac{1}{\Gamma(\alpha)} \sum_{s=t}^{b-1}(\sigma(s)-t)^{\overline{\alpha-1}}f(s),~~t \in ~_{b-1}\mathbb{N}.
\end{equation}
\end{defn}

Regarding the delta left fractional sum we observe the following:

(i) $\Delta_a^{-\alpha}$ maps functions defined on $\mathbb{N}_a$ to
functions defined on $\mathbb{N}_{a+\alpha}$.

(ii) $u(t)=\Delta_a^{-n}f(t)$, $n \in \mathbb{N}$, satisfies the
initial value problem
\begin{equation} \label{ivpf}
\Delta^n u(t)=f(t),~~t\in N_a,~u(a+j-1)=0,~ j=1,2,...,n.
\end{equation}

(iii) The Cauchy function $\frac{(t-\sigma(s))^{(n-1)}}{(n-1)!}$
vanishes at $s=t-(n-1),...,t-1$.

\indent

Regarding the delta right fractional sum we observe the following:

(i)  $~_{b}\Delta^{-\alpha}$ maps functions defined on $_{b}\mathbb{N}$ to
functions defined on $_{b-\alpha}\mathbb{N}$.

(ii) $u(t)=~_{b}\Delta^{-n}f(t)$, $n \in \mathbb{N}$, satisfies the
initial value problem
\begin{equation} \label{ivpb}
\nabla_\ominus^n u(t)=f(t),~~t\in~ _{b}N,~u(b-j+1)=0,~ j=1,2,...,n.
\end{equation}

(iii) the Cauchy function $\frac{(\rho(s)-t)^{(n-1)}}{(n-1)!}$
vanishes at $s=t+1,t+2,...,t+(n-1)$.

\indent

Regarding the nabla left fractional sum we observe the following:

(i) $ \nabla_a^{-\alpha}$ maps functions defined on $\mathbb{N}_a$ to functions defined on $\mathbb{N}_{a}$.

(ii)$ \nabla_a^{-n}f(t)$ satisfies the n-th order discrete initial value problem

\begin{equation}\label{s1}
    \nabla^n y(t)=f(t),~~~\nabla^i y(a)=0,~~i=0,1,...,n-1
\end{equation}

(iii) The Cauchy function $\frac{(t-\rho(s))^{\overline{n-1}}}{\Gamma(n)}$ satisfies $\nabla^n y(t)=0$.

\indent

Regarding the nabla right fractional sum we observe the following:

(i) $ ~_{b}\nabla^{-\alpha}$ maps functions defined on $~_{b}\mathbb{N}$ to functions defined on $~_{b}\mathbb{N}$.

(ii)$ ~_{b}\nabla^{-n}f(t)$ satisfies the n-th order discrete initial value problem

\begin{equation}\label{s2}
    ~_{\ominus}\Delta^n y(t)=f(t),~~~ ~_{\ominus}\Delta^i y(b)=0,~~i=0,1,...,n-1.
\end{equation}
The proof can be done inductively. Namely, assuming it is true for $n$, we have
\begin{equation}\label{t1}
    ~_{\ominus}\Delta^{n+1} ~_{b}\nabla^{-(n+1)}f(t)=~_{\ominus}\Delta^{n}[-\Delta ~_{b}\nabla^{-(n+1)}f(t)].
\end{equation}

By the help of (\ref{oper3}), it follows that
\begin{equation}\label{t2}
  ~_{\ominus}\Delta^{n+1} ~_{b}\nabla^{-(n+1)}f(t)=  ~_{\ominus}\Delta^{n} ~_{b}\nabla^{-n}f(t)=f(t).
\end{equation}
The other part is clear by using the convention that $\sum_{k=s+1}^s=0$.

(iii) The Cauchy function $\frac{(s-\rho(t))^{\overline{n-1}}}{\Gamma(n)}$ satisfies $_{\ominus}\Delta^n y(t)=0$.

\indent

\begin{defn} \label{fractional differences}
(i)\cite{Miller}  The (delta) left fractional difference of order $\alpha>0$ (starting from $a$ ) is defined by:
\begin{equation}\label{dls}
    \Delta_a^{\alpha} f(t)=\Delta^n \Delta_a^{-(n-\alpha)} f(t)= \frac{\Delta^n}{\Gamma(n-\alpha)} \sum_{s=a}^{t-(n-\alpha)}(t-\sigma(s))^{(n-\alpha-1)}f(s),~~t \in \mathbb{N}_{a+(n-\alpha)}
\end{equation}

(ii) \cite{Thabet2} The (delta) right fractional difference of order $\alpha>0$ (ending at  $b$ ) is defined by:
\begin{equation}\label{drd}
   ~_{b}\Delta^{\alpha} f(t)=  \nabla_{\circleddash}^n  ~_{b}\Delta^{-(n-\alpha)}f(t)=\frac{(-1)^n \nabla ^n}{\Gamma(n-\alpha)} \sum_{s=t+(n-\alpha)}^{b}(s-\sigma(t))^{(n-\alpha-1)}f(s),~~t \in ~_{b-(n-\alpha)}\mathbb{N}
\end{equation}

(iii) \cite{Holm} The (nabla) left fractional difference of order $\alpha>0$ (starting from $a$ ) is defined by:
\begin{equation}\label{nld}
  \nabla_a^{\alpha} f(t)=\nabla^n \nabla_a^{-(n-\alpha)}f(t)= \frac{\nabla^n}{\Gamma(n-\alpha)} \sum_{s=a+1}^t(t-\rho(s))^{\overline{n-\alpha-1}}f(s),~~t \in \mathbb{N}_{a+1}
\end{equation}

(iv) (\cite{THFer}, \cite{Thsh} The (nabla) right fractional difference of order $\alpha>0$ (ending at $b$ ) is defined by:
\begin{equation}\label{nrd}
   ~_{b}\nabla^{\alpha} f(t)= ~_{\circleddash}\Delta^n ~_{b}\nabla^{-(n-\alpha)}f(t) =\frac{(-1)^n\Delta^n}{\Gamma(n-\alpha)} \sum_{s=t}^{b-1}(s-\rho(t))^{\overline{n-\alpha-1}}f(s),~~t \in ~ _{b-1}\mathbb{N}
\end{equation}

\end{defn}

Regarding the domains of the fractional type differences we observe:

(i) The delta left fractional difference $\Delta_a^\alpha$ maps functions defined on $\mathbb{N}_a$ to functions defined on $\mathbb{N}_{a+(n-\alpha)}$.

(ii) The delta right fractional difference $~_{b}\Delta^\alpha$ maps functions defined on $~_{b}\mathbb{N}$ to functions defined on $~_{b-(n-\alpha)}\mathbb{N}$.

(iii) The nabla left fractional difference $\nabla_a^\alpha$ maps functions defined on $\mathbb{N}_a$ to functions defined on $\mathbb{N}_{a+n}$ .

(iv)  The nabla right fractional difference $~_{b}\nabla^\alpha$ maps functions defined on $~_{b}\mathbb{N}$ to functions defined on $~_{b-n}\mathbb{N}$ .

\begin{lem} \label{left dual}\cite{Ferd3}
Let $0\leq n-1< \alpha \leq n$ and let $y(t)$ be defined on $\mathbb{N}_a$. Then the following statements are valid.

(i)$ (\Delta_a^\alpha) y(t-\alpha)= \nabla_{a-1}^\alpha y(t)$ for $t \in \mathbb{N}_{n+a}$.

(ii) $ (\Delta_a^{-\alpha}) y(t+\alpha)= \nabla_{a-1}^{-\alpha} y(t)$ for $t \in \mathbb{N}_a$.
\end{lem}

\begin{lem} \label{right dual} \cite{Thsh}
Let $y(t)$ be defined on $~_{b+1}\mathbb{N}$. Then the following statements are valid.

(i)$ (~_{b}\Delta^\alpha) y(t+\alpha)= ~_{b+1}\nabla^\alpha y(t)$ for $t \in ~_{b-n}\mathbb{N}$.

(ii) $ (~_{b}\Delta^{-\alpha}) y(t-\alpha)= ~_{b+1}\nabla^{-\alpha} y(t)$ for $t \in ~_{b}\mathbb{N}$.
\end{lem}

If $f(s)$ is defined on $N_a\cap ~_{b}N$ and $a\equiv b~ (mod
~1)$ then $(Qf)(s)=f(a+b-s)$. The Q-operator generates a dual identity by which the left type and the right type fractional sums and differences are related. Using the change of variable $u=a+b-s$, in \cite{Thabet1} it was shown  that
\begin{equation}\label{sum pr}
    \Delta_a^{-\alpha}Qf(t)= Q~_{b}\Delta^{-\alpha}f(t),
\end{equation}

and hence
\begin{equation}\label{pr}
    \Delta_a^\alpha Qf(t)= (Q ~_{b}\Delta^\alpha f)(t).
\end{equation}
The proof of (\ref{pr}) follows by the definition, (\ref{sum pr}) and by noting that

$$-Q\nabla f(t)=\Delta Qf(t).$$

Similarly, in the nabla case we have
\begin{equation}\label{npr sum}
   \nabla_a^{-\alpha}Qf(t)= Q~_{b}\nabla^{-\alpha}f(t),
 \end{equation}

and hence

\begin{equation}\label{rpr}
   \nabla_a^\alpha Qf(t)=( Q ~_{b}\nabla^\alpha f)(t).
\end{equation}

The proof of (\ref{rpr}) follows by the definition, (\ref{npr sum}) and that

$$-Q\Delta f(t)=\nabla Qf(t).$$
For more details about the discrete version of the Q-operator we refer to \cite{Thsh}.

From the difference calculus or time scale calculus, for a natural $n$ and a sequence $f$, we recall

\begin{equation}\label{delta d}
    \Delta^n f(t)=\sum_{k=0}^n (-1)^{k}(\begin{array}{c}
                                  n \\
                                  k
                                \end{array})f(t+n-k),
\end{equation}
and
\begin{equation}\label{nabla d}
    \nabla^n f(t)=\sum_{k=0}^n (-1)^{k}(\begin{array}{c}
                                  n \\
                                  k
                                \end{array})f(t-n+k),
\end{equation}

\section{The fractional differences and sums with binomial coefficients}
We first give the definition of fractional order of (\ref{delta d})  and (\ref{nabla d}) in the left and right sense.
\begin{defn}
The (binomial) delta left fractional difference and sum of order $\alpha >0$ for a function $f$ on defined on $\mathbb{N}_a$, are defined by
\begin{itemize}
  \item \textbf{(a)}$$B\Delta_a^\alpha= \sum_{k=0}^{\alpha+t-a}(-1)^k (\begin{array}{c}
                                                           \alpha \\
                                                           k
                                                         \end{array}
  ) f(t+\alpha-k),~~t \in \mathbb{N}_{a+n-\alpha}$$

  \item \textbf{(b)}$$B\Delta_a^{-\alpha}= \sum_{k=0}^{\alpha+t-a}(-1)^k (\begin{array}{c}
                                                          - \alpha \\
                                                           k
                                                         \end{array}
  ) f(t-\alpha-k),~~t \in \mathbb{N}_{a+\alpha},$$
\end{itemize}
where $(-1)^k (\begin{array}{c}
                                                           \alpha \\
                                                           k
                                                         \end{array})= (\begin{array}{c}
                                                           \alpha+k-1 \\
                                                           k
                                                         \end{array})$.
\end{defn}

\begin{defn}
The (binomial) nabla left fractional difference and sum of order $\alpha >0$ for a function $f$ on defined on $\mathbb{N}_a$, are defined by
\begin{itemize}
  \item \textbf{(a)}$$B\nabla_a^\alpha= \sum_{k=0}^{t-a-1}(-1)^k (\begin{array}{c}
                                                           \alpha \\
                                                           k
                                                         \end{array}
  ) f(t-k),~~t \in \mathbb{N}_{a+n}$$

  \item \textbf{(b)}$$B\nabla_a^{-\alpha}= \sum_{k=0}^{t-a-1}(-1)^k (\begin{array}{c}
                                                          - \alpha \\
                                                           k
                                                         \end{array}
  ) f(t-k),~~t \in \mathbb{N}_{a}.$$
\end{itemize}

\end{defn}
Analogously, in the right case we can define:

\begin{defn}
The (binomial) delta right fractional difference and sum of order $\alpha >0$ for a function $f$  defined on $_{b}\mathbb{N}$, are defined by
\begin{itemize}
  \item \textbf{(a)}$$~_{b}\Delta B^\alpha= \sum_{k=0}^{\alpha+b-t}(-1)^k (\begin{array}{c}
                                                           \alpha \\
                                                           k
                                                         \end{array}
  ) f(t-\alpha+k),~~t \in {b-n+\alpha}\mathbb{N}$$

  \item \textbf{(b)}$$~_{b}\Delta B^{-\alpha}= \sum_{k=0}^{-\alpha+b-t}(-1)^k (\begin{array}{c}
                                                           -\alpha \\
                                                           k
                                                         \end{array}
  ) f(t+\alpha+k),~~t \in {b-\alpha}\mathbb{N}$$

\end{itemize}

\end{defn}

\begin{defn}
The (binomial) nabla right fractional difference and sum of order $\alpha >0$ for a function $f$ on defined on $_{b}\mathbb{N}$, are defined by
\begin{itemize}
  \item \textbf{(a)}$$~_{b}\nabla B^\alpha= \sum_{k=0}^{b-t-1}(-1)^k (\begin{array}{c}
                                                           \alpha \\
                                                           k
                                                         \end{array}
  ) f(t-k),~~t \in ~_{b-n}\mathbb{N}$$

  \item \textbf{(b)}$$~_{b}\nabla B^{-\alpha}= \sum_{k=0}^{b-t-1}(-1)^k (\begin{array}{c}
                                                           -\alpha \\
                                                           k
                                                         \end{array}
  ) f(t-k),~~t \in~ _{b}\mathbb{N}.$$

\end{itemize}

\end{defn}

We next proceed to show that the Riemann fractional differences and sums coincide with the binomial ones defined above. We will use the dual identities in Lemma \ref{left dual} and Lemma \ref{right dual}, and the action of the discrete  version of the Q-operator to follow easy proofs and verifications.
In \cite{Holm} the author used a delta Leibniz's Rule to obtain the following alternative definition for Riemann delta left fractional differences:

\begin{equation} \label{a}
    \Delta_a^\alpha f(t)=\frac{1}{\Gamma(-\alpha)} \sum_{s=a}^{t+\alpha} (t-\sigma(s))^{(-\alpha-1)},~~\alpha \notin \mathbb{N}, ~~t \in \mathbb{N}_{a+n-\alpha},
\end{equation}
then proceeded with long calculations and showed, actually, that

\begin{equation}\label{ldelta}
    \Delta_a^\alpha f(t)= B\Delta_a^\alpha f(t), ~~~\Delta_a^{-\alpha} f(t)= B\Delta_a^{-\alpha} f(t)
\end{equation}

\begin{thm} \label{equiv}
Let $f$ be defined on suitable domains and $\alpha >0$. Then
\begin{itemize}
  \item  1)$$\Delta_a^\alpha f(t)= B\Delta_a^\alpha f(t), ~~~\Delta_a^{-\alpha} f(t)= B\Delta_a^{-\alpha} f(t)$$
  \item 2)$$~_{b}\Delta^\alpha f(t)= ~_{b}\Delta B^\alpha f(t), ~~~ ~_{b}\Delta ^{-\alpha} f(t)= \Delta B_a^{-\alpha} f(t)$$
  \item 3)$$\nabla_a^\alpha f(t)= B\nabla_a^\alpha f(t), ~~~\nabla_a^{-\alpha} f(t)= B\nabla_a^{-\alpha} f(t)$$
  \item 4) $$~_{b}\nabla^\alpha f(t)= ~_{b}\nabla B^\alpha f(t), ~~~~_{b}\nabla^{-\alpha} f(t)= ~_{b}\nabla B^{-\alpha} f(t)$$
\end{itemize}

\end{thm}

\begin{proof}
\begin{itemize}
  \item 1) follows by (\ref{ldelta}).
  \item 2)By the discrete Q-operator action we have $$ ~_{b}\Delta^\alpha f(t)= Q \Delta_a (Qf)(t)=Q \sum_{k=0}^{\alpha+t-a} (-1)^k(\begin{array}{c}
                                                           \alpha \\
                                                           k
                                                         \end{array}
  ) (Qf)(t+\alpha-k)=~_{b}\Delta B^\alpha f(t).$$ The fractional sum part is also done in a similar way by using the Q-operator.
  \item 3) By the dual identity in Lemma \ref{left dual} (i) and (\ref{ldelta}), we have $$\nabla_a^\alpha f(t)=\Delta_{a+1}^\alpha f(t+\alpha)=B\Delta_{a+1}^\alpha f(t+\alpha)=B\nabla_a^\alpha f(t).$$ The fractional sum part can be proved similarly by using Lemma \ref{left dual} (ii) and (\ref{ldelta}).
  \item 4) The proof can be achieved by either 2) and Lemma \ref{right dual} or alternatively, by 3) and the discrete  Q-operator.
\end{itemize}
\end{proof}

\begin{rem}
In analogous to (\ref{a}) the authors in \cite{Ahrendt} used a nabla Leibniz's Rule to prove that

\begin{equation}\label{aa}
    \nabla_a^\alpha f(t)= \frac{1}{\Gamma(-\alpha)} \sum_{s=a+1}^t (t-\rho(s))^{\overline{-\alpha-1}}f(s).
\end{equation}

In \cite{THFer} the authors used a delta Leibniz's Rule to prove  the following formula for nabla right fractional differences
\begin{equation}\label{b}
    ~_{b}\nabla^\alpha f(t)= \frac{1}{\Gamma(-\alpha)} \sum_{s=t}^{b-1} (s-\rho(t))^{\overline{-\alpha-1}}f(s).
\end{equation}
Similarly, we can use a nabla Leibniz's Rule to prove the following formula for the delta right fractional differences:
\begin{equation}\label{bb}
    ~_{b}\Delta^\alpha f(t)= \frac{1}{\Gamma(-\alpha)} \sum_{s=t-\alpha}^{b} (s-\sigma(t))^{(-\alpha-1)}f(s).
\end{equation}
We here remark that the proofs of the last three parts of Theorem \ref{equiv} can be done  alternatively by proceeding as in \cite{Holm} starting from (\ref{aa}),(\ref{b})and (\ref{bb}). Also it is worth mentioning that mixing both delta and nabla operators in defining delta and nabla right  Riemann fractional differences was essential in proceeding, through the dual identities and the discrete Q-operator or delta and nabla type Leibniz's Rules, to obtain the main results in this article  \cite{Thsh}.
\end{rem}


\begin{thebibliography}{99}

\bibitem{book1}Samko G. Kilbas A. A., Marichev, Fractional Integrals
and Derivatives: Theory and Applications, Gordon and Breach, Yverdon, 1993.

\bibitem{book2}
I. Podlubny, Fractional Differential Equations, Academic Press: San
Diego CA, (1999).

\bibitem{book3} Kilbas A., Srivastava M. H.,and Trujillo J. J., Theory and Application of Fractional Differential Equations, North
Holland Mathematics Studies 204, 2006.

\bibitem{book4} B.J. West, M. Bologna, P. Grogolini, Physics of Fractal Operators. Springer, New
York, 2003.

\bibitem{book5} R.L. Magin, Fractional Calculus in Bioengineering, Begell House Publisher, Inc.
Connecticut, 2006.
\bibitem{book6}
 D. Baleanu, K. Diethelm, E. Scalas, J.J. Trujillo, Fractional Calculus Models and Numerical
 Methods (Series on Complexity, Nonlinearity and Chaos. World Scientific), 2012.

\bibitem{Gray} H. L. Gray and N. F.Zhang, On a new definition of the fractional difference, Mathematics of Computaion 50, (182), 513-529 (1988.)

\bibitem{Miller} K. S. Miller,  Ross B.,Fractional difference
calculus, \emph{Proceedings of the International Symposium on
Univalent Functions, Fractional Calculus and Their Applications}, Nihon University, Koriyama, Japan, (1989), 139-152.

\bibitem{Ferd1} F.M. At{\i}c{\i}  and Eloe P. W.,  A Transform method in
discrete fractional calculus, \emph{International Journal of
Difference Equations}, vol 2, no 2, (2007), 165--176.

\bibitem{Ferd2}F.M.  At{\i}c{\i}  and Eloe P. W.,  Initial value problems in
discrete fractional calculus, \emph{Proceedings of the American
Mathematical Society}, 137, (2009), 981-989.

\bibitem{Ferd3} F. M. At{\i}c{\i} and Paul W.Eloe, Discrete fractional
calculus with the nabla operator, Electronic Journal of Qualitative Theory
of Differential equations, Spec. Ed. I, 2009 No.3,1--12.

\bibitem{Ferd4}F.M. At{\i}c{\i},  \c{S}eng\"{u}l S., Modelling with fractional difference equations,Journal of Mathematical Analysis and Applications,
369 (2010) 1-9.

\bibitem{Gronwall} F. M. At{\i}c{\i}, Paul W. Eloe, Gronwall's inequality on discrete fractional calculus, Computerand Mathematics with applications, In Press, doi:10.1016/camwa. 2011.11.029.

\bibitem{Thabet1} T. Abdeljawad , On Riemann and Caputo fractional differences,
Computers and Mathematics with Applications, Volume 62, Issue 3, August 2011, Pages 1602-1611.

\bibitem{Thabet2}T. Abdeljawad , D. Baleanu , Fractional Differences and
integration by parts, Journal of Computational Analysis and Applications
vol 13 no. 3 , 574-582 (2011).

\bibitem{Holm} M. Holm, The theory of discrete fractional calculus development and application, Dissertation 2011.

\bibitem{Gdelta}G. A. Anastassiou,Principles of delta fractional calculus on time scales and inequalities, Mathematical and Computer Modelling, 52 (2010)556-566.
\bibitem{Gnabla} G. A. Anastassiou, Nabla discrete calcilus and nabla inequalities, Mathematical and Computer Modelling, 51 (2010) 562-571.

\bibitem{Gfound} G. A. Anastassiou,Foundations of nabla fractional calculus on time scales and inequalities,Computer and Mathematics with Applications, 59 (2010) 3750-3762.

\bibitem{Nuno} Nuno R. O. Bastos, Rui A. C. Ferreira, Delfim F. M.
Torres, Discrete-time fractional variational problems, Signal Processing,  91(3): 513-524 (2011).

\bibitem{Adv}Bohner M.  and A. Peterson, Advances in Dynamic Equations on Time Scales, Birkh\"{a}user, Boston, 2003.

\bibitem{Boros}G. Boros and V. Moll, Iresistible Integrals; Symbols,Analysis and Expreiments in the Evaluation of Integrals,
Cambridge University PressCambridge 2004.

\bibitem{Grah} R. L. Graham, D. E. Knuth and O. Patashnik, Concrete Mathematics, A Foundation for Copmuter Science,
2nd ed. , Adison-Wesley,Reading, MA, 1994.

\bibitem{Spanier} J. Spanier and K. B. Oldham, The Pochhammer Polynomials $(x)_n$, An Atlas of Functions,
pp. 149-156,Hemisphere, Washington, DC, 1987.


\bibitem{THFer}T. Abdeljawad, F. At{\i}c{\i}, On the Definitions of Nabla Fractional Operators, Abstract and Applied Analysis, in press.

\bibitem{Ahrendt}K. Ahrendt, L. Castle, M. Holm, K. Yochman, Laplace transforms for the nabla -difference operator and a fractional variation of parameters formula, \emph{Communications in Applied Analysis}, to appear.
\bibitem{Thsh}T. Abdeljawad, Dual identities in fractional difference calculus within Riemann, submitted.
\end{thebibliography}
\end{document}